\documentclass[english]{article}
\usepackage[margin=2.5cm]{geometry}
\usepackage[hyphens]{url}
\usepackage{hyperref}
\usepackage[hyphenbreaks]{breakurl}
\usepackage{amsthm,amsfonts,amssymb,amsmath}
\usepackage{graphicx}

\newtheorem{theorem}{Theorem}[section]
\newtheorem{proposition}[theorem]{Proposition}
\newtheorem{lemma}[theorem]{Lemma}
\newtheorem{conjecture}[theorem]{Conjecture}

\newtheorem{definition}[theorem]{Definition}
\newtheorem{problem}[theorem]{Problem}

\theoremstyle{definition}

\numberwithin{equation}{section}
\numberwithin{figure}{section}
\allowdisplaybreaks

\usepackage{babel}
\usepackage[T1]{fontenc}
\usepackage[utf8]{inputenc}

\newcommand{\su}{\subseteq}
\newcommand{\sm}{\setminus}

\newcommand{\F}{\mathbb{F}}

\newcommand{\spn}{\operatorname{span}}

\begin{document}
\title{Rota's basis conjecture holds for random bases of vector spaces}
\author{Lisa Sauermann\thanks{Department of Mathematics, Massachusetts Institute of Technology, Cambridge, MA.
Email: \href{lsauerma@mit.edu}{\nolinkurl{lsauerma@mit.edu}}. Research supported by NSF Award DMS-1953772.}}

\maketitle

\begin{abstract}\noindent
In 1989, Rota conjectured that, given $n$ bases $B_1,\dots,B_n$ of the vector space $\mathbb{F}^n$ over some field $\mathbb{F}$, one can always decompose the multi-set $B_1\cup \dots \cup B_n$ into transversal bases. This conjecture remains wide open despite of a lot of attention. In this paper, we consider the setting of random bases $B_1,\dots,B_n$. More specifically, our first result shows that Rota's basis conjecture holds with probability $1-o(1)$ as $n\to \infty$ if the bases $B_1,\dots,B_n$ are chosen independently uniformly at random among all bases of $\mathbb{F}_q^n$ for some finite field $\mathbb{F}_q$ (the analogous result is trivially true for an infinite field $\mathbb{F}$). In other words, the conjecture is true for almost all choices of bases $B_1,\dots,B_n\su \F_q^n$. Our second, more general, result concerns random bases $B_1,\dots,B_n\su S^n$ for some given finite subset $S\su \F$ (in other words, bases $B_1,\dots,B_n$ where all vectors have entries in $S$). We show that when choosing bases $B_1,\dots,B_n\su S^n$ independently uniformly at random among all bases that are subsets of $S^n$, then again Rota's basis conjecture holds with probability $1-o(1)$ as $n\to \infty$.
\end{abstract}

\section{Introduction}

Rota's basis conjecture (see \cite[Conjecture 4]{huang-rota}) is a famous conjecture from 1989 concerning bases in vector spaces. Given $n$ bases $B_1,\dots,B_n$ of an $n$-dimensional vector space, the conjecture asserts that one can  decompose the multi-set $B_1\cup \dots \cup B_n$ into bases of the vector space that are \emph{transversal} with respect to the original bases $B_1,\dots,B_n$. Here, a basis $B\su B_1\cup \dots \cup B_n$ is called transversal with respect to $B_1,\dots,B_n$ if it contains exactly one vector from each of $B_1,\dots,B_n$ (where $B_1,\dots,B_n$ are interpreted as subsets of the multi-set $B_1\cup \dots \cup B_n$ in the natural way, and if a vector $v$ appears in two or more of the bases $B_1,\dots,B_n$, then the multiple copies of $v$ in the multi-set $B_1\cup \dots \cup B_n$ are distinguished by which set $B_i$ they came from). Note that, as each basis $B_1,\dots,B_n$ consists of $n$ vectors, the union $B_1\cup \dots \cup B_n$ has size $n^2$. Each transversal basis consists again of $n$ vectors, so a decomposition of  $B_1\cup \dots \cup B_n$ into transversal bases must have exactly $n$ transversal bases. Since every $n$-dimensional vector space over any field $\F$ is isomorphic to $\F^n$, the conjecture can be restated as follows.

\begin{conjecture}[Rota's basis conjecture]\label{conjecture-rota}
Let $\F$ be a field and let $B_1,\dots,B_n\su \F^n$ be bases of the vector space $\F^n$. Then the multi-set $B_1\cup \dots\cup B_n$ can be partitioned into $n$ bases of $\F^n$, which are transversal with respect to the original bases $B_1,\dots,B_n$.
\end{conjecture}

Rota's basis conjecture is also often considered in the more general setting of matroids rather than vector spaces. Despite of a lot of attention (including a ``Polymath'' project dedicated to the conjecture, see \cite{polymath}), the conjecture remains wide open, even in the setting of vector spaces.

Drisko \cite{drisko} and Glynn \cite{glynn} proved that Conjecture \ref{conjecture-rota} is true over fields $\F$ of characteristic zero if $n-1$ or $n+1$, respectively, is a prime number (more, precisely Drisko and Glynn proved the Alon-Tarsi conjecture concerning enumerations of cerrtain types of Latin squares for such $n$, and via earlier work of Huang and Rota \cite{huang-rota}, this implies Conjecture \ref{conjecture-rota} in these cases). Furthermore, the matroid version of Conjecture \ref{conjecture-rota} has been proved for certain special classes of matroids (specifically, for paving matroids \cite{geelen-humphries} and for strongly base orderable matroids \cite{wild}). Aharoni and Berger \cite[Theorem 10.4]{aharoni-berger} showed that the natural fractional relaxation of the conjecture holds, even in the matroid setting.

There have also been many works concerning partial decomposition or covering versions of the conjecture. Specifically, Buci\'{c}, Kwan, Pokrovskiy, and Sudakov \cite{bucic-et-al} (improving earlier results in \cite{dong-geelen, geelen-webb}) proved that one can find $n/2-o(n)$ disjoint transversal bases in $B_1\cup \dots \cup B_n$, Pokrovskiy \cite{prokrovskiy} proved that one can find $n-o(n)$ disjoint linearly independent transversal sets of size $n-o(n)$ in $B_1\cup \dots \cup B_n$, and Aharoni and Berger \cite[Assertion 8.11]{aharoni-berger} proved that $B_1\cup \dots \cup B_n$ can be decomposed into $2n$ disjoint linearly independent transversal sets (and the bound $2n$ has been slightly improved to $2n-2$ by Polymath, see \cite{polymath}).

In this paper, we consider the setting of random bases $B_1,\dots,B_n$. More precisely, in our first result, we assume that the bases $B_1,\dots,B_n$ are chosen independently uniformly at random among all bases of $\F^n$. If the field $\F$ is infinite, then with probability $1$, all $n^2$ vectors appearing in $B_1,\dots,B_n$ are in generic position, and Conjecture \ref{conjecture-rota} is trivially true in this case (as any decomposition of $B_1\cup \dots \cup B_n$ into $n$ transversal sets will give the desired transversal bases). However, this argument does not apply over finite fields, and in particular not if $n$ is large with respect to the field size. The problem of proving Rota's basis conjecture for random bases over some fixed finite field has been suggested by Ferber \cite{ferber}. Resolving this problem, we prove that the conjecture indeed holds asymptotically almost surely for independent uniformly random bases $B_1,\dots,B_n\su \F_q^n$ for any (finite) field size $q$.

\begin{theorem}\label{thm-1}
Fix a prime power $q\geq 2$. Let $B_1,\dots,B_n\su \F_q^n$ be independent uniformly random bases of the vector space $\F_q^n$ (each chosen uniformly at random among all bases of $\F_q^n$).  Then, with probability $1-o(1)$ as $n\to \infty$, the multi-set $B_1\cup \dots\cup B_n$ can be partitioned into $n$ bases of $\F_q^n$, which are transversal with respect to the original bases $B_1,\dots,B_n$.
\end{theorem}

We remark that in fact, the $o(1)$-term here can be taken independently of $q$ (so Theorem \ref{thm-1} also holds if $q$ is allowed to depend on $n$). Note that Theorem \ref{thm-1} can be interpreted as saying that Rota's basis conjecture holds for ``almost all'' choices of bases $B_1,\dots,B_n\su \F_q^n$.

For various problems concerning vectors in $\mathbb{R}^n$, or more generally in $\F^n$, it is of interest to restrict one's attention to vectors with entries in some particular fixed set, most notably $\{0,1\}$-vectors. Our next result is a version of Theorem \ref{thm-1}, where we only consider bases of $\F^n$ consisting of vectors with entries from some specified set $S$. In other words, we prove that Rota's basis conjecture holds for random bases of vectors with entries from some specified set (e.g.\ $\{0,1\}$-vectors).

\begin{theorem}\label{thm-2}
Fix a field $\F$ and a finite subset $S\su \F$ of size $|S|\geq 2$. Let $B_1,\dots,B_n\su S^n$ be independent random bases of the vector space $\F^n$, where each $B_i$ is chosen uniformly at random from the collection $\{B\su S^n\mid B\text{ is a basis of }\F^n\}$ of the bases of $\F^n$ with all vector entries in $S$.  Then, with probability $1-o(1)$ as $n\to \infty$, the multi-set $B_1\cup \dots\cup B_n$ can be partitioned into $n$ bases of $\F^n$, which are transversal with respect to the original bases $B_1,\dots,B_n$.
\end{theorem}

We remark that the $o(1)$-term here can be taken independently of $\F$ and $S$.

Note that Theorem  \ref{thm-1} can be viewed as a special case of Theorem \ref{thm-2} by taking $\F=S=\F_q$. The more general setting of Theorem \ref{thm-2} introduces various challenges caused by the lack of symmetry between different bases $B\su S^n$ of $\F^n$. Indeed, in the setting of Theorem  \ref{thm-1}, any two bases $B\su \F_q^n$ can be transformed into each other by an isomorphism of $\F^n$. However, such an isomorphism does not necessarily preserve the set $S^n$ in the setting of Theorem \ref{thm-2}. Here is another example illustrating this point: In $\F_q^n$, each non-zero vector is contained in the same number of bases $B$ of $\F_q^n$. However, taking $S=\{0,1,2\}$ and a prime $q\geq 5$, not every non-zero vector in $\{0,1,2\}^n$ is contained in the same number of bases $B\su \{0,1,2\}^n$ of $\F_q^n$ (if $n\geq 2$).

One may also be interested in special sets of ``allowed vectors'' other than sets of the form $S^n$ as in Theorem \ref{thm-2}. For example, graphic matroids naturally corresponds to the setting of considering vectors in $\{1,0,-1\}^n$ consisting of precisely one $1$-entry and one $(-1)$-entry (with the remaining $n-2$ entries being zero). So it is natural to study the setting where for some subset $T\su \F^n$ of ``allowed vectors'' one considers independent uniformly random bases $B_1,\dots,B_n\su T$ of $\F^n$ (i.e. independent random bases, each chosen uniformly at random among all bases of $\F^n$ that are subsets of $T$). Taking $T=S^n$ gives precisely the setting of Theorem \ref{thm-2}. Our arguments to prove Theorem \ref{thm-2} generalize to sets $T\su \F^n$ that are reasonably well spread out over $\F^n$, in the sense of being not too ``clumped'' on any subspace of $\F^n$. The following definition makes this condition precise.

\begin{definition}\label{defi-dispersed}
For $0<c<1$, we say that a set $T\su \F^n$ of vectors in the vector space $\F^n$ over some field $\F$ is \emph{$c$-dispersed} if every (linear) subspace $V\su \F^n$ satisfies $|V\cap T|\leq c^{n-\dim V}\cdot |T|$.
\end{definition}

In other words, a set $T\su \F^n$  is $c$-dispersed if every subspace $V\su \F^n$ contains at most a  $(c^{n-\dim V})$-fraction of the vectors in $T$. Note that every $c$-dispersed set $T\su \F^n$ is also $c'$-dispersed if $0<c<c'<1$. Furthermore, we have $\spn(T)=\F^n$ for every $c$-dispersed set $T\su \F^n$ (for any $0<c<1$).

Our arguments for proving Theorem \ref{thm-2} give an analogous result with $S^n$ replaced by a $c$-dispersed set $T\su \F^n$ for some fixed $0<c<1$. This is stated in the following theorem.

\begin{theorem}\label{thm-3}
Fix a field $\F$ and $0<c<1$. Let $T\su \F^n$ be $c$-dispersed and let $B_1,\dots,B_n\su T$ be independent random bases of the vector space $\F^n$, where each $B_i$ is chosen uniformly at random from the collection $\{B\su T\mid B\text{ is a basis of }\F^n\}$.  Then, with probability $1-o(1)$ as $n\to \infty$, the multi-set $B_1\cup \dots\cup B_n$ can be partitioned into $n$ bases of $\F^n$, which are transversal with respect to the original bases $B_1,\dots,B_n$.
\end{theorem}

The $o(1)$-term here can actually be taken independently of $\F$ (but it is required for $c$ to be fixed).

\textit{Organization.} This paper is organized as follows. In Section \ref{sect-aux}, we deduce Theorems \ref{thm-1} and \ref{thm-2} from Theorem \ref{thm-3}, and state some auxiliary lemmas for the proof of Theorem \ref{thm-3}. Section \ref{sect-main-proof} explains the approach for proving Theorem \ref{thm-3}, but the main part of the proof is (encapsulated in a certain proposition) postponed to Section \ref{sect-proof-propo}. Section \ref{sect-prob-dist} contains some preparations for the proof in Section \ref{sect-proof-propo}, while Section \ref{sect-lemmas} contains the proofs of some lemmas in Section \ref{sect-proof-propo}. At the end of the paper, we make some concluding remarks in Section~\ref{sect-concluding-remarks}.

\textit{Acknowledgements.} The author is grateful for helpful discussions with Asaf Ferber and Matthew Kwan.

\section{Auxiliary lemmas and preparations}
\label{sect-aux}

\begin{lemma}\label{lemma-dispersedness-S-n}
For any field $\F$ and any subset $S\su \F$ of size $|S|\geq 2$, the set $S^n\su \F^n$ is $(1/|S|)$-dispersed.
\end{lemma}
\begin{proof}
Let $V\su \F^n$ be a subspace of dimension $k$. We need to show that $|V\cap S^n|\leq (1/|S|)^{n-k}\cdot |S^n|$, i.e.\ that $|V\cap S^n|\leq |S|^k$. The $k$-dimensional subspace $V\su \F^n$ is given by a system of $n-k$ linearly independent linear equations. When writing this system in row echelon form, we obtain $k$ free variables (whereas the other $n-k$ variables are determined by linear expressions in the $k$ free variables). For a vector in $V\cap S^n$, all free variables need to take values in $S$, so there are at most $|S|^k$ choices for the free variables (which then, in turn, determine the remaining variables). Thus, $|V\cap S^n|\leq |S|^k$ as desired. 
\end{proof}

Using Lemma \ref{lemma-dispersedness-S-n}, we can easily deduce Theorem \ref{thm-2} from Theorem \ref{thm-3}.

\begin{proof}[Proof of Theorem \ref{thm-2} assuming Theorem \ref{thm-3}]
Let $S\su \F$ be a subset of size $|S|\geq 2$. By Lemma \ref{lemma-dispersedness-S-n}, the set $T=S^n\su \F^n$ is $(1/|S|)$-dispersed, and in particular $(1/2)$-dispersed. We can now apply Theorem \ref{thm-3} with $c=1/2$ to obtain the desired conclusion.
\end{proof}

Recalling that Theorem \ref{thm-2} immediately implies Theorem \ref{thm-1} as a special case (by taking $\F=S=\F_q$), it only remains to prove Theorem \ref{thm-3}. This will be our goal for the rest of this paper.

So let us from now on fix a field $\F$ and $0<c<1$. All of our asymptotic $o$-notation in the rest of this paper may depend on the fixed value of $c$.

We define
\begin{equation}\label{eq-defi-c-prime}
c'=(1-c)\cdot (1-c^2)\cdot (1-c^3)\dotsm = \prod_{i=1}^{\infty}(1-c^i).
\end{equation}
It is clear that this infinite product is well-defined, since the sequence of partial products $\prod_{i=1}^{m}(1-c^i)$ for $m\to \infty$ is a monotone decreasing sequence of non-negative numbers and must therefore be convergent. Note that $c'$ only depends on our fixed value of $c$.

\begin{lemma} We have $0<c'<1$.
\end{lemma}
\begin{proof}
Since $c>0$, it is easy to see that $c'<1$. In order to show $c'>0$, choose a sufficiently large positive integer $\ell$ such that $c^\ell<1-c$. Then we have
\[c'=\prod_{i=1}^{\ell-1}(1-c^i)\cdot \prod_{i=\ell}^{\infty}(1-c^i)\geq \prod_{i=1}^{\ell-1}(1-c^i)\cdot (1-c^\ell-c^{\ell+1}-c^{\ell+2}-\dots)=\prod_{i=1}^{\ell-1}(1-c^i)\cdot \left(1-\frac{c^\ell}{1-c}\right)>0,\]
noting that the last term is a finite product of positive factors.
\end{proof}

Finally, we will use the following auxiliary lemma in the proof of Theorem \ref{thm-3}. We postpone the proof of this lemma to Section \ref{sect-lemmas} (the proof is not difficult, but involves a somewhat lengthy computation).

\begin{lemma}\label{lem-auxiliary}
Let $Z_1,\dots,Z_K$ be a sequence of random variables. For some finite set $H$, consider events $\mathcal{E}_h^{(\ell)}$ for $h\in H$ and $\ell=1,\dots,K$, where each event $\mathcal{E}_h^{(\ell)}$ depends only on the outcomes of $Z_1,\dots, Z_\ell$. Suppose that $\beta_1,\dots,\beta_K$ with $0\leq \beta_\ell\leq 1$ for $\ell=1,\dots,K$ are chosen such that the following condition holds: For any subset $H'\su H$, any $\ell=1,\dots,K$ and any outcomes of $Z_1,\dots, Z_{\ell-1}$, we have (subject to the randomness of $Z_\ell$ when conditioning on the given outcomes of $Z_1,\dots, Z_{\ell-1}$) that
\[\Pr\left[\mathcal{E}_h^{(\ell)}\textnormal{ holds for all }h\in H'\,\Big\vert \,Z_1,\dots,Z_{\ell-1}\right]\leq (\beta_\ell)^{|H'|}.\]
Then we can conclude that
\[\Pr\left[\textnormal{for each }h\in H\textnormal{ there is }\ell\in \{1,\dots,K\}\textnormal{ such that }\mathcal{E}_h^{(\ell)}\textnormal{ holds}\right]\leq \left(1-(1-\beta_1)\dotsm (1-\beta_K)\right)^{|H|}.\]
\end{lemma}

\section{Proof strategy for Theorem \ref{thm-3}}
\label{sect-main-proof}

Recall that in the previous section we fixed a field $\F$ and $0<c<1$. As in Theorem \ref{thm-3}, let $T\su \F^n$ be a $c$-dispersed subset.

We now need to consider independent random bases $B_1,\dots,B_n\su T$ be of the vector space $\F^n$, where each $B_i$ is chosen independently uniformly at random among all bases of $\F^n$ that are subsets of $T$. For each basis $B_1,\dots,B_n$, we can imagine that it is equipped with an ordering of its $n$ vectors (take such an ordering uniformly randomly among all $n!$ orderings). In other words, we can write $B_i=\{b_{i,1},\dots,b_{i,n}\}$ for $i=1,\dots,n$ where $b_{i,1},\dots,b_{i,n}\in T$. Now, for each $i=1,\dots,n$ the $n$-tuple $(b_{i,1},\dots,b_{i,n})\in T^n$ is uniformly random among all $n$-tuples in $T^n$ that are bases of $\F^n$ (and these $n$-tuples are independent for different $i$).

In other words, let us from now on consider independent random $n$-tuples $(b_{i,1},\dots,b_{i,n})\in T^n$  for $i=1,\dots,n$, where each $(b_{i,1},\dots,b_{i,n})\in T^n$ is chosen uniformly random among all $n$-tuples in $T^n$ that are bases of $\F^n$, and let $B_i=\{b_{i,1},\dots,b_{i,n}\}$ for $i=1,\dots,n$.

In order to prove Theorem \ref{thm-3}, we need to prove that with probability $1-o(1)$, the multi-set $B_1\cup \dots\cup B_n$ can be partitioned into $n$ bases of $\F^n$, which are transversal with respect to the original bases $B_1,\dots,B_n$ (i.e. which are each of the form $\{b_{1,j_1},b_{2,j_2},\dots,b_{n,j_n}\}$ for some $j_1,\dots,j_n\in \{1,\dots,n\}$).

In order to construct the desired transversal bases (with high probability), let us consider multi-sets $X_j$ and $Y_j$ for $j=1,\dots,n$ defined as follows: First fix some choice of $n'\in \{\lfloor n/2\rfloor,\lceil n/2\rceil\}$ throughout the rest of the paper. Now, for every $j=1,\dots,n$, define $X_j=\{b_{1,j},b_{2,j},\dots,b_{n',j}\}$ and $Y_j=\{b_{n'+1,j},b_{n'+2,j},\dots,b_{n,j}\}$ (taken as a multi-set in case there are repetitions among the vectors $b_{1,j},b_{2,j},\dots,b_{n',j}$ and $b_{n'+1,j},b_{n'+2,j},\dots,b_{n,j}$, respectively). Note that $X_1,\dots,X_n$ form a partition of the multi-set $B_1\cup\dots\cup B_{n'}$, and $Y_1,\dots,Y_n$ form a partition of the multi-set $B_{n'+1}\cup\dots\cup B_{n}$. Thus, $X_1,\dots,X_n$ and $Y_1,\dots,Y_n$ together form a partition of $B_1\cup\dots\cup B_{n}$.

Note that for any $h,j\in \{1,\dots,n\}$, the multi-set $X_h\cup Y_{j}=\{b_{1,h},\dots,b_{n',h},b_{n'+1,j},\dots,b_{n,j}\}$ is transversal with respect to the bases $B_1,\dots,B_n$. Hence, if the $n$ vectors in $X_h\cup Y_{j}$ are linearly independent, then $X_h\cup Y_{j}$ is a transversal basis. Let us define a bipartite graph, with $n$ vertices on the left and $n$ vertices on the right, where for all $h,j\in \{1,\dots,n\}$ we draw an edge between vertex $j$ on the left and vertex $h$ on the right if and only if the $n$ vectors in $X_h\cup Y_{j}$ are linearly independent. Then each edge in this graph corresponds to a transversal basis of the form $X_h\cup Y_{j}$.

Hence, in order to find a partition of $B_1\cup \dots\cup B_n$ into $n$ transversal bases, it suffices to find a perfect matching in this bipartite graph. Indeed, each edge of such a perfect matching would give a transversal basis $X_h\cup Y_{j}$ and together they would form a partition of the multi-set $X_1\cup \dots\cup X_n\cup Y_1\cup \dots\cup Y_n=B_1\cup \dots\cup B_n$.

Thus, in order to prove Theorem \ref{thm-3}, it suffices to show that the bipartite graph defined above has a perfect matching with probability $1-o(1)$. This follows from the following two propositions, which assert that with probability $1-o(1)$ one can find perfect matchings separately for the ``top half'' and the ``bottom half'' of the graph.

\begin{proposition}\label{propo-1}
Let $m\in \{\lfloor n/2\rfloor,\lceil n/2\rceil\}$, and consider the induced subgraph of the bipartite graph defined above, where we only take the vertices $1,\dots,m$ on the left and the vertices $1,\dots,m$ on the right. Then this induced subgraph has a perfect matching with probability $1-o(1)$.
\end{proposition}

\begin{proposition}\label{propo-2}
Let $m\in \{\lfloor n/2\rfloor,\lceil n/2\rceil\}$, and consider the induced subgraph of the bipartite graph defined above, where we only take the vertices $m+1,\dots,n$ on the left and the vertices $m+1,\dots,n$ on the right. Then this induced subgraph has a perfect matching with probability $1-o(1)$.
\end{proposition}

Note that the sets of vertices considered in the two propositions form a partition of the vertex set of the original bipartite graph. Hence, if we can find perfect matchings in the two induced subgraphs considered in these propositions, then we obtain a perfect matching in the original bipartite graph. Thus, Propositions \ref{propo-1} and \ref{propo-2} together imply Theorem \ref{thm-3}.

So it suffices to prove Propositions \ref{propo-1} and \ref{propo-2}. Note that these two propositions are equivalent to each other upon reversing the order of the vectors $b_{i,1},\dots,b_{i,n}$ in each basis $B_i$ (because that reverses the order of the list of sets $X_1,\dots,X_n$ and also of the list of sets $Y_1,\dots,Y_n$).

Thus, it actually suffices to just prove Proposition \ref{propo-1}, and the rest of this paper is devoted to proving this proposition.

The reason for considering the induced subgraph in Proposition \ref{propo-1} rather than the entire original bipartite graph is as follows. We will imagine that for each $i=1,\dots,n$ we expose the vectors $b_{i,1}, \dots,b_{i,m}$ one vector at a time in this order. The choices for the vectors $b_{i,1}, \dots,b_{i,m}$ are clearly not independent of each other (for example, the vectors $b_{i,1}, \dots,b_{i,m}$ must always be linearly independent for each $i=1,\dots,n$). However, the fact that $m$ is significantly smaller than $n$ makes the dependence of each of the vectors $b_{i,1}, \dots,b_{i,m}$ on the previously exposed vectors  easier to handle (in particular, because the span of the previously exposed vectors is not too large).

In the next section, we will analyze the probability distribution for each new vector when choosing the vectors $b_{i,1}, \dots,b_{i,m}$ one at a time for each $i=1,\dots,n$ (note that the choices for different $i$ are actually independent of each other since the bases $B_1,\dots,B_n$ are independent). The result of this section will then be used in our proof of Proposition \ref{propo-1} in Section \ref{sect-proof-propo}.

\section{Lemmas for the probability distribution for new basis vectors}
\label{sect-prob-dist}

Let us imagine that each random basis $B_i=\{b_{i,1},\dots,b_{i,n}\}\su T$ is exposed one vector at a time in this order. Then each vector $b_{i,j}$ for $1\leq j\leq n$ is some random vector in $T$ whose distribution depends on the previously exposed vectors $b_{i,1},\dots, b_{i,j-1}$ of the basis $B_i$ (recall that the different bases $B_1,\dots,B_n$ are independent). Given $b_{i,1},\dots, b_{i,j-1}$, the new vector $b_{i,j}\in T$ must be linearly independent from $b_{i,1},\dots, b_{i,j-1}$. Note, however, that the distribution of $b_{i,j}$ is not necessarily uniform among all vectors in $T\setminus \spn(b_{i,1},\dots, b_{i,j-1})$. Instead, the different possibilities of $b_{i,j}\in T\setminus \spn(b_{i,1},\dots, b_{i,j-1})$ have probabilities proportional to the number of possibilities for extending $b_{i,1},\dots, b_{i,j-1},b_{i,j}$ further to a basis of $\F^n$ with vectors in $T$. 

The next two lemmas show useful properties of this probability distribution. Recall that in (\ref{eq-defi-c-prime}) we defined $c'>0$ (depending only on $0<c<1$, which we fixed).

\begin{lemma}\label{lemma-probabilities-differ-factor-c-prime}
For some $i,j\in \{1,\dots, n\}$, assume that the vectors $b_{i,1},\dots, b_{i,j-1}\in T$ have already been exposed, and fix any outcomes of these vectors. Now, consider the probability distribution for $b_{i,j}$ in the set $T\setminus \spn(b_{i,1},\dots, b_{i,j-1})$. Then for any two vectors in $T\setminus \spn(b_{i,1},\dots, b_{i,j-1})$ the corresponding probabilities differ by a factor of at most $1/c'$.
\end{lemma}

\begin{proof}
For any $x,y\in T\setminus \spn(b_{i,1},\dots, b_{i,j-1})$, we need to show $\Pr[b_{i,j}=x]\geq c'\cdot \Pr[b_{i,j}=y]$, conditioned on the given outcomes of $b_{i,1},\dots, b_{i,j-1}$ (which are linearly independent). Recall that the probability for $b_{i,j}$ to attain a certain vector is proportional to the number of possibilities of for extending the resulting sequence of vectors $b_{i,1},\dots, b_{i,j-1},b_{i,j}$ to a basis of $\F^n$ with vectors in $T$. Hence the desired inequality is equivalent to
\begin{multline*}
|\{(z_{j+1},\dots,z_n)\in T^{n-j}\mid b_{i,1},\dots,b_{i,j-1},x,z_{j+1},\dots,z_n\text{ is a basis of }\F^n\}|\\
\geq c'\cdot|\{(z_{j+1},\dots,z_n)\in T^{n-j}\mid b_{i,1},\dots,b_{i,j-1},y,z_{j+1},\dots,z_n\text{ is a basis of }\F^n\}|
\end{multline*}
The right-hand side can is clearly at most $c'\cdot |T|^{n-j}$. Let us now show that the left-hand side is at least $c'\cdot |T|^{n-j}$.

In order to show this, consider independent uniformly random vectors $z_{j+1},\dots,z_n\in T^{n-j}$. We wish to show that with probability at least $c'$ the vectors $b_{i,1},\dots,b_{i,j-1},x,z_{j+1},\dots,z_n$ form a basis of $\F^n$. This happens if and only if the vectors $b_{i,1},\dots,b_{i,j-1},x,z_{j+1},\dots,z_n$ are linearly independent. Note that $b_{i,1},\dots,b_{i,j-1},x$ are linearly independent (since $b_{i,1},\dots, b_{i,j-1}$ must be linearly independent and $x\in T\setminus \spn(b_{i,1},\dots, b_{i,j-1})$). Now, the probability that the (uniformly random) vector $z_{j+1}\in T$ lies in $\spn(b_{i,1},\dots, b_{i,j-1},x)$ is at most
\[\frac{|T\cap \spn(b_{i,1},\dots, b_{i,j-1},x)|}{|T|}\leq c^{n-\dim \spn(b_{i,1},\dots, b_{i,j-1},x)}=c^{n-j},\]
where the inequality holds since $T$ is $c$-dispersed (see Definition \ref{defi-dispersed}). Hence the vectors $b_{i,1},\dots,b_{i,j-1},x, z_{j+1}$ are linearly independent with probability at least $1-c^{n-j}$. Whenever this happens, we can repeat the same argument for $z_{j+2}$, obtaining that $z_{j+2}$ lies in $\spn(b_{i,1},\dots, b_{i,j-1},x,z_{j+1})$ with probability at most $c^{n-j-1}$. So with probability at least $(1-c^{n-j})\cdot (1-c^{n-j-1})$, the vectors $b_{i,1},\dots,b_{i,j-1},x, z_{j+1}, z_{j+2}$ are linearly independent. Repeating this argument, we can show inductively that for every $h=j+1,\dots,n$, the vectors $b_{i,1},\dots,b_{i,j-1},x, z_{j+1}, \dots, z_{h}$ are linearly independent with probability at least $(1-c^{n-j})\dotsm (1-c^{n+1-h})$. Thus, taking $h=n$, we see that the vectors $b_{i,1},\dots,b_{i,j-1},x,z_{j+1},\dots,z_n$ are linearly independent with probability at least
\[(1-c^{n-j})\cdot (1-c^{n-j-1})\dotsm (1-c)=\prod_{\ell=1}^{n-j}(1-c^\ell)\geq \prod_{\ell=1}^{\infty}(1-c^\ell)=c'.\]

Thus, we proved that
\begin{multline*}
|\{(z_{j+1},\dots,z_n)\in T^{n-j}\mid b_{i,1},\dots,b_{i,j-1},x,z_{j+1},\dots,z_n\text{ is a basis of }\F^n\}|\\
\geq c'\cdot |T|^{n-j}\geq c'\cdot|\{(z_{j+1},\dots,z_n)\in T^{n-j}\mid b_{i,1},\dots,b_{i,j-1},y,z_{j+1},\dots,z_n\text{ is a basis of }\F^n\}|,
\end{multline*}
as desired.
\end{proof}

The next lemma only applies to $b_{i,j}$ with $j\leq \lceil n/2\rceil$ (and only with sufficiently large $n$ in terms of $c$), but note that in order to prove Proposition \ref{propo-1} we only need to consider the vectors $b_{i,j}$ with $j\leq m\leq \lceil n/2\rceil$ (since these are the only vectors appearing in the sets $X_1,\dots,X_m,Y_1,\dots,Y_m$).

\begin{lemma}\label{lemma-new-vector-in-subspace}
Suppose that $n$ is sufficiently large with respect to $c$ such that $c^{n/2}\leq (1-c)/2$. For some $i,j\in \{1,\dots, n\}$ with $j\leq \lceil n/2\rceil$, assume that the vectors $b_{i,1},\dots, b_{i,j-1}\in T$ have already been exposed. For some $1\leq k\leq n$, let $V\su \F^n$ be a subspace of dimension $n-k$. Then, conditional on any fixed outcomes of $b_{i,1},\dots, b_{i,j-1}$, the random vector $b_{i,j}$ satisfies $b_{i,j}\in V$ with probability 
\[\Pr\left[b_{i,j}\in V\mid b_{i,1},\dots, b_{i,j-1}\right]\leq \left(1+(c'/2)\cdot \frac{1-c^k}{c^k}\right)^{-1}.\]
\end{lemma}
\begin{proof}
As in the lemma statement, let us fix any outcomes $b_{i,1},\dots, b_{i,j-1}$, and let $W=\spn(b_{i,1},\dots, b_{i,j-1})$. Since $j\leq \lceil n/2\rceil$, we have $\dim W\leq n/2$ and Definition \ref{defi-dispersed} yields $|W\cap T|\leq c^{n-\dim W}\cdot |T|\leq c^{n/2}\cdot |T|$.

Since the vectors $b_{i,1},\dots, b_{i,j}$ must always be linearly independent, we have $b_{i,j}\in T\setminus W$ for any outcome of $b_{i,j}$ (conditioned on the fixed outcomes of $b_{i,1},\dots, b_{i,j-1}$). By Lemma \ref{lemma-probabilities-differ-factor-c-prime}, for any two vectors in $x,y\in T\setminus W$, we have
\[\Pr[b_{i,j}=x\mid b_{i,1},\dots, b_{i,j-1}]\geq c'\cdot \Pr[b_{i,j}=y\mid b_{i,1},\dots, b_{i,j-1}].\]
Let $\rho$ be the maximum value of $\Pr[b_{i,j}=y\mid b_{i,1},\dots, b_{i,j-1}]$ among all $y\in T\setminus W$. Then for every $x\in T\setminus W$, we have
\[c'\cdot \rho\leq \Pr[b_{i,j}=x\mid b_{i,1},\dots, b_{i,j-1}]\leq \rho.\]
Hence
\[\Pr[b_{i,j}\in V\mid b_{i,1},\dots, b_{i,j-1}]=\sum_{x\in V\cap (T\sm W)}\Pr[b_{i,j}=x\mid b_{i,1},\dots, b_{i,j-1}]\leq |V\cap T|\cdot \rho,\]

On the other hand,
\[\Pr[b_{i,j}\not\in V\mid b_{i,1},\dots, b_{i,j-1}]=\sum_{x\in T\sm (V\cup W)}\Pr[b_{i,j}=x\mid b_{i,1},\dots, b_{i,j-1}]\geq |T\sm (V\cup W)|\cdot c'\cdot \rho.\]
Using that
\[|T\sm (V\cup W)|\geq |T|-|V\cap T|-|W\cap T|\geq |T|-|V\cap T|-c^{n/2}\cdot |T|=(1-c^{n/2})|T|-|V\cap T|,\]
this implies
\[\Pr[b_{i,j}\not\in V\mid b_{i,1},\dots, b_{i,j-1}]\geq \left((1-c^{n/2})|T|-|V\cap T|\right)\cdot c'\cdot \rho.\]
Now, using that $|V\cap T|\leq c^k\cdot |T|$ (by Definition \ref{defi-dispersed} for the $(n-k)$-dimensional space $V$), we can conclude
\begin{align*}
\frac{1}{\Pr[b_{i,j}\in V\mid b_{i,1},\dots, b_{i,j-1}]}&=\frac{\Pr[b_{i,j}\in V\mid b_{i,1},\dots, b_{i,j-1}]+\Pr[b_{i,j}\not\in V\mid b_{i,1},\dots, b_{i,j-1}]}{\Pr[b_{i,j}\in V\mid b_{i,1},\dots, b_{i,j-1}]}\\
&=1+\frac{\Pr[b_{i,j}\not\in V\mid b_{i,1},\dots, b_{i,j-1}]}{\Pr[b_{i,j}\in V\mid b_{i,1},\dots, b_{i,j-1}]}\\
&\geq 1+\frac{\left((1-c^{n/2})|T|-|V\cap T|\right)\cdot c'\cdot \rho}{|V\cap T|\cdot \rho}\\
&=1+c'\cdot \left(\frac{(1-c^{n/2})|T|}{|V\cap T|}-1\right)\\
&\geq 1+c'\cdot \left(\frac{(1-c^{n/2})|T|}{c^k|T|}-1\right)\\
&=1+c'\cdot \frac{1-c^{n/2}-c^k}{c^k}.
\end{align*}
By our assumption $k\geq 1$, we have $c^{n/2}\leq (1-c)/2\leq (1-c^k)/2$ and obtain
\[\frac{1}{\Pr[b_{i,j}\in V\mid b_{i,1},\dots, b_{i,j-1}]}\geq 1+c'\cdot \frac{1-c^{n/2}-c^k}{c^k}\geq 1+c'\cdot \frac{(1-c^k)/2}{c^k}=1+(c'/2)\cdot \frac{1-c^k}{c^k}.\]
This gives the desired inequality.
\end{proof}

The term on the right-hand side of the inequality in Lemma \ref{lemma-new-vector-in-subspace} will occur repeatedly (for different values of $k$) in our proof. To simplify notation, let us write
\begin{equation}\label{eq-definition-alpha}
\alpha_k=\left(1+(c'/2)\cdot \frac{1-c^k}{c^k}\right)^{-1}
\end{equation}
for any positive integer $k$. Note that $0<\alpha_k<1$ (recalling that $0<c<1$ and $0<c'<1$) and that $\alpha_k$ depends only on $c$ and $k$. Now, the bound in Lemma \ref{lemma-new-vector-in-subspace} reads $\Pr\left[b_{i,j}\in V\mid b_{i,1},\dots, b_{i,j-1}\right]\leq \alpha_k$.
Furthermore, note that for every positive integer $k$ we have
\begin{equation}\label{eq-upper-bound-alpha-k}
\alpha_k=\left(1+(c'/2)\cdot \frac{1-c^k}{c^k}\right)^{-1}= \left(1-(c'/2)+ \frac{c'/2}{c^k}\right)^{-1}< \left(\frac{c'/2}{c^k}\right)^{-1}=\frac{2}{c'}\cdot c^k,
\end{equation}
and also note that $\alpha_1\geq \alpha_2\geq \alpha_3\geq \dots$ is a monotone decreasing sequence.

\begin{lemma}\label{lem-alpha-powers}
For any positive integers $k$ and $\ell$, we have $\alpha_{k\ell}\leq (\alpha_k)^{\ell}$.
\end{lemma}
\begin{proof}
Note that
\begin{multline*}
\left(c^{k}+(c'/2)\cdot (1-c^k)\right)^{\ell}=c^{k\ell}+\sum_{i=1}^{\ell}\binom{\ell}{i}(c'/2)^i\cdot (1-c^k)^i\cdot (c^k)^{\ell-i}\\
\leq c^{k\ell}+(c'/2)\cdot\sum_{i=1}^{\ell}\binom{\ell}{i} (1-c^k)^i\cdot (c^k)^{\ell-i}=c^{k\ell}+(c'/2)\cdot (1-c^{k\ell}),
\end{multline*}
where we used that $c'/2<1$. Dividing by $c^{k\ell}$ yields
\[\left(1+(c'/2)\cdot \frac{1-c^k}{c^k}\right)^{\ell}\leq 1+(c'/2)\cdot \frac{1-c^{k\ell}}{c^{k\ell}},\]
which upon taking inverses gives the desired inequality $(\alpha_k)^{\ell}\geq \alpha_{k\ell}$.
\end{proof}

Let us now define
\begin{equation}\label{eq-defi-delta}
\delta=\frac{1}{3}\cdot (1-\alpha_1)\cdot (1-\alpha_2)\cdot (1-\alpha_3)\dotsm = \frac{1}{3}\cdot \prod_{k=1}^{\infty}(1-\alpha_k).
\end{equation}
This infinite product is well-defined, since the corresponding sequence of partial products is a monotone decreasing sequence of non-negative numbers and therefore converges. Note that $\delta$ only depends on $c$.

Before starting the proof of Proposition \ref{propo-1} in the following section, we first establish that $\delta>0$.

\begin{lemma}\label{lemma-delta-positive}
We have $0<\delta<1/3$.
\end{lemma}
\begin{proof}
Since $0<\alpha_k<1$ for all positive integers $k$, it is easy to see that $0\leq \delta<1/3$. In order to establish that $\delta>0$, choose a sufficiently large positive integer $\ell$ such that $c^\ell<(c'/2)\cdot (1-c)$. Then by (\ref{eq-upper-bound-alpha-k}) we have
\begin{multline*}
\delta=\frac{1}{3}\cdot\prod_{k=1}^{\ell-1}(1-\alpha_k)\cdot \prod_{i=k}^{\infty}(1-\alpha_k)\geq \frac{1}{3}\cdot\prod_{k=1}^{\ell-1}(1-\alpha_k)\cdot (1-\alpha_\ell-\alpha_{\ell+1}-\dots)\\
\geq \frac{1}{3}\cdot\prod_{k=1}^{\ell-1}(1-\alpha_k)\cdot\left(1-\frac{2}{c'}\cdot (c^{\ell}+c^{\ell+1}+\dots)\right)=\frac{1}{3}\cdot\prod_{k=1}^{\ell-1}(1-\alpha_k)\cdot \left(1-\frac{2}{c'}\cdot \frac{c^\ell}{1-c}\right)>0,
\end{multline*}
noting that the last term is a finite product of positive factors.
\end{proof}

\section{Finding the perfect matching: proof of Proposition \ref{propo-1}}
\label{sect-proof-propo}

Let us take $m\in \{\lfloor n/2\rfloor,\lceil n/2\rceil\}$ as in Proposition \ref{propo-1}. Recall that $0<c<1$ is fixed. Let us assume that $n$ is sufficiently large with respect to $c$, such that in particular $c^{n/2}\leq (1-c)/2$ (note that then we can apply Lemma \ref{lemma-new-vector-in-subspace} whenever $j\leq m$).

Let $G$ be the graph in Proposition \ref{propo-1}, i.e.\ $G$ is the bipartite graph with $m$ vertices on the left and $m$ vertices on the right, where for all $h,j\in \{1,\dots,m\}$ we draw an edge between vertex $h$ on the left and vertex $j$ on the right if and only if the $n$ vectors in $X_h\cup Y_{j}$ are linearly independent. We need to prove that with probability $1-o(1)$ (as $n\to \infty$) this graph $G$ has a perfect matching.

Our first lemma states that with high probability each of the multi-sets $X_1,\dots, X_m$ individually is linearly independent in $\F^n$ (and is hence in particular just an ordinary set of vectors with no repetitions).

\begin{lemma}\label{lem-no-isolated-vert}
With probability $1-o(1)$, each of the multi-sets $X_1,\dots, X_m$ is linearly independent.
\end{lemma}

\begin{proof}
It suffices to prove that for each $j\in \{1,\dots,m\}$, the probability that $X_j$ is not linearly independent is at most $o(1/n)$. Indeed, then a union bound over all $j\in \{1,\dots,m\}$ shows that with probability at least $1-m\cdot o(1/n)=1-o(1)$, each of $X_1,\dots, X_m$ is linearly independent.

So let $j\in \{1,\dots,m\}$. We prove that the probability that $X_j$ is not linearly independent is at most $o(1/n)$, even when conditioning on any outcomes of the sets $X_1,\dots,X_{j-1}$, i.e.\ of the vectors $b_{i,1},\dots,b_{i,j-1}$ for $i=1,\dots,n'$. So let us fix any outcomes of $b_{i,1},\dots,b_{i,j-1}$ for $i=1,\dots,n'$.

Let us now expose the vectors of $X_j=\{b_{1,j},\dots,b_{n',j}\}$ one vector at a time in this order. If $X_j$ is not linearly independent, then one of the vectors $b_{i,j}$ for some $i\in \{1,\dots,n'\}$ must be in the span of the previously exposed vectors $b_{1,j},\dots,b_{i-1,j}$. For each $i\in \{1,\dots,n'\}$, this span has dimension at most $i-1\leq n'-1$, so by Lemma \ref{lemma-new-vector-in-subspace} and (\ref{eq-definition-alpha}) the probability that $b_{i,j}$ is inside $\spn(b_{1,j},\dots,b_{i-1,j})$ is at most $\alpha_{n-n'+1}\leq (2/c')\cdot c^{n-n'+1}\leq (2/c')\cdot c^{n/2}$ (here, we used (\ref{eq-upper-bound-alpha-k}) and $n'\leq \lceil n/2\rceil$). Thus, the total probability that $X_j$ is not linearly independent is indeed at most $n'\cdot (2/c')\cdot c^{n/2}\leq n \cdot (2/c')\cdot c^{n/2} = o(1/n)$ (recalling that $0<c<1$).
\end{proof}

Recall that we need to prove that  the bipartite graph $G$ has a perfect matching with probability $1-o(1)$. In order to do so, we will show that that with probability $1-o(1)$ the graph $G$ satisfies the condition in Hall's marriage theorem. To establish this, we first show that with probability $1-o(1)$ all vertices in $G$ have high degree (see Lemmas \ref{lem-high-degrees-left} and \ref{lem-high-degrees-right} below). And second, we show for some $L$ (depending only on $c$), that with probability $1-o(1)$, for any $L$ distinct vertices on the left side the union of their neighborhoods is large (see Lemma \ref{lem-L-tuple-nbhd} below).

Recall that we defined $\delta$ (depending only on $c$) in (\ref{eq-defi-delta}) and we established in Lemma  \ref{lemma-delta-positive} that $\delta>0$. The following lemma is a key step towards our our first goal of showing that with probability $1-o(1)$ all vertices in $G$ have high degree. The lemma states, roughly speaking, that for every vertex $h$ on the left, each vertex $j$ on the right has an edge to $h$ on the left with probability at least $3\delta$ (in other words, $X_h\cup Y_j$ is linearly independent with probability at least $3\delta$). More precisely, the lemma states that this is true even when conditioning on the outcomes of the sets $Y_1,\dots,Y_{j-1}$.

\begin{lemma}\label{lem-prob-edge}
Let $h,j\in \{1,\dots,m\}$, and fix any outcome of $X_h=\{b_{1,h},\dots,b_{n',h}\}$ such that $b_{1,h},\dots,b_{n',h}$ are linearly independent. Furthermore, consider any fixed outcomes of the vectors in the sets $Y_1,\dots,Y_{j-1}$, i.e.\ of the vectors $b_{i,1},\dots,b_{i,j-1}$ for all $i=n'+1,\dots,n$. Then, conditional on the outcomes of the vectors in $Y_1,\dots,Y_{j-1}$, subject to the randomness of  $Y_j=\{b_{n'+1,j},\dots,b_{n,j}\}$, the multi-set $X_h\cup Y_j$ is linearly independent with probability at least $3\delta$.
\end{lemma}

\begin{proof}
Let us expose the vectors of $Y_j=\{b_{n'+1,j},\dots,b_{n,j}\}$ one vector at a time. Note that these random vectors are actually probabilistically independent of each other (since the different bases $B_i$ are independent). By assumption, $X_h=\{b_{1,h},\dots,b_{n',h}\}$ is linearly independent. Hence $\spn (X_h)$ is a subspace of dimension $n'$, and by Lemma \ref{lemma-new-vector-in-subspace} and (\ref{eq-definition-alpha}) with probability at least $1-\alpha_{n-n'}$ the vector $b_{n'+1,j}$ (when conditioning on the given outcomes of $b_{n'+1,1},\dots ,b_{n'+1,j-1}$) is outside this subspace and hence linearly independent from $X_h$. Assuming that this hold, $\spn (X_h\cup \{b_{n'+1,j}\})$ is a subspace of dimension $n'+1$. Then with probability at least $1-\alpha_{n-n'-1}$ the vector $b_{n'+2,j}$ (when conditioning on the given outcomes of $b_{n'+2,1},\dots ,b_{n'+2,j-1}$) is outside $\spn (X_h\cup \{b_{n'+1,j}\})$ and hence linearly independent from $X_h\cup \{b_{n'+1,j}\}$. Continuing like this, we can show that for every $\ell=n'+1,\dots,n$, the multi-set $X_h\cup \{b_{n'+1,j},\dots,b_{\ell,j}\}$ (when conditioning on the given outcomes of $b_{i,1},\dots,b_{i,j-1}$ for all $i=n'+1,\dots,\ell$) is linearly independent with probability at least $(1-\alpha_{n-n'})(1-\alpha_{n-n'-1})\dotsm (1-\alpha_{n-\ell+1})$. In particular, for $\ell=n$ we obtain that $X_h\cup Y_j$ is linearly independent with probability at least
\[(1-\alpha_{n-n'})(1-\alpha_{n-n'-1})\dotsm (1-\alpha_{1})= \prod_{k=1}^{n-n'} (1-\alpha_k)\geq \prod_{k=1}^{\infty} (1-\alpha_k)=3\delta,\]
when conditioning on the given outcomes of $b_{i,1},\dots,b_{i,j-1}$ for all $i=n'+1,\dots,n$.
\end{proof}

As an easy corollary of Lemma \ref{lem-prob-edge} and Lemma \ref{lem-no-isolated-vert}, we obtain the following lemma.

\begin{lemma}\label{lem-high-degrees-left}
With probability $1-o(1)$, in the graph $G$ every vertex on the left has degree at least $2\delta m$.
\end{lemma}

\begin{proof}
Let us first expose the bases $B_1,\dots,B_{n'}$, which determine the sets $X_1,\dots,X_m$. By Lemma \ref{lem-no-isolated-vert}, with probability $1-o(1)$ each of the multi-sets $X_1,\dots, X_m$ is linearly independent. So let us fix an outcome of $X_1,\dots, X_m$, where each of these sets is linearly independent.

It now suffices to prove that, subject to the randomness of $Y_1,\dots,Y_m$, for each $h\in \{1,\dots,m\}$ with probability at most $o(1/n)$ vertex $h$ on the left has degree less than $2\delta m$. Indeed, then by a union bound, with probability at most $m\cdot o(1/n)=o(1)$, there is a vertex on the left with degree less than $2\delta m$.

So consider some $h\in \{1,\dots,m\}$. By Lemma \ref{lem-prob-edge}, subject to the randomness of $Y_1$, there is an edge between vertex $h$ on the left and vertex $1$ on the right with probability at least $3\delta$. After exposing $Y_1$ and conditioning on its outcome, subject to the randomness of $Y_2$, by Lemma \ref{lem-prob-edge} there is an edge between vertex $h$ on the left and vertex $2$ on the right with probability at least $3\delta$. Continuing this, we see that for every $j\in \{1,\dots,m\}$, when conditioning on any outcomes of $Y_1,\dots,Y_{j-1}$, subject to the randomness of $Y_j$, there is an edge between vertex $h$ on the left and vertex $j$ on the right with probability at least $3\delta$. Thus, subject to the randomness of $Y_1,\dots,Y_m$, the degree of vertex $h$ on the left is a random variable that stochastically dominates a binomial random variable $Z\sim \operatorname{Bin}(m,3\delta)$. Hence the probability that vertex $h$ on the left has degree less than $2\delta m$ is at most
\[ \Pr[Z<2\delta m]\leq \exp\left(-2\cdot \frac{(3\delta m-2\delta m)^2}{m}\right)=e^{-2\delta^2 m} \leq e^{-\delta^2\cdot (n-1)}=o(1/n),\]
as desired. Here, in the first step we used the Chernoff bound (see e.g.\ \cite[Theorem A.1.4]{alon-spencer}) and in the third step we used that $m\in \{\lfloor n/2\rfloor,\lceil n/2\rceil\}$.
\end{proof}

Analogously to Lemma \ref{lem-high-degrees-left}, we can also show that with high probability, every vertex on the right has degree at least $2\delta m$.

\begin{lemma}\label{lem-high-degrees-right}
With probability $1-o(1)$, in the graph $G$ every vertex on the right has degree at least $2\delta m$.
\end{lemma}

\begin{proof}
Note that our setup is completely symmetric in the left and the right side. So Lemma \ref{lem-high-degrees-right} follows in a completely analogous way to Lemma \ref{lem-high-degrees-left} (by first stating and proving the analogous versions of Lemmas \ref{lem-no-isolated-vert} and  \ref{lem-prob-edge} with left and right side interchanged).
\end{proof}

Let us now fix a sufficiently large positive integer $L$ depending only on $c$ such that
\begin{equation}\label{eq-def-L}
(1-3\delta)^L\leq \delta/2
\end{equation}
(recall that $0<\delta<1/3$ only depends on $c$).  Furthermore, let us fix a sufficiently large positive integer $K$ depending only on $c$ such that
\[L\cdot \frac{2}{c'}\cdot \frac{c^K}{1-c}\leq \frac{\delta}{2}.\]
Recall that we are assuming that $n$ is sufficiently large with respect to $c$, so we can in particular assume that $n\geq 2K$ (and therefore $\lfloor n/2\rfloor\geq K$). Note that by (\ref{eq-upper-bound-alpha-k}), we have
\begin{equation}\label{eq-summing-up-alpha}
\alpha_K+\alpha_{K+1}+\alpha_{K+2}+\dots< \frac{2}{c'}\cdot (c^K+c^{K+1}+c^{K+2}+\dots)=\frac{2}{c'}\cdot \frac{c^K}{1-c}\leq \frac{\delta}{2L}.
\end{equation}

Our second goal is now to show that with probability $1-o(1)$, for any $L$ distinct vertices on the left side of $G$, the union of their neighborhoods is large. When showing this, we have to take care of the possibility that certain subspaces that appear in the argument might have unusually large intersections. The following definition captures these ``bad'' situations that would be problematic for our argument.

For $j\in \{1,\dots,m\}$ and $k\in \{1,\dots,K\}$, let us define $Y_j^{(k)}=\{b_{n'+1,j},\dots,b_{n-k,j}\}$. In other words, $Y_j^{(k)}$ is obtained from $Y_j=\{b_{n'+1,j},\dots,b_{n,j}\}$ by omitting the last $k$ vectors.

\begin{definition}\label{defi-bad}
For a subset $H\su \{1,\dots,m\}$ of size $1\leq |H|\leq L$, as well as $j\in \{1,\dots,m\}$ and $k\in \{1,\dots,K\}$, let us say that $(H,Y_j^{(k)})$ is \emph{bad} if
\[\dim\left(\bigcap_{h\in H}\spn(X_h\cup Y_j^{(k)})\right)>n-k\cdot |H|.\]
Furthermore, let us say that $(H,j)$ is \emph{bad} if $(H,Y_j^{(k)})$ is bad for some $k\in \{1,\dots,K\}$.
\end{definition}

The next lemma states that with high probability there are no bad pairs $(H,j)$.

\begin{lemma}\label{lem-no-bad-set}
With probability $1-o(1)$, for every subset $H\su \{1,\dots,m\}$ of size $1\leq |H|\leq L$ and every $j\in \{1,\dots,m\}$ the pair $(H,j)$ is not bad.
\end{lemma}

We postpone the proof of Lemma \ref{lem-no-bad-set} to Section \ref{sect-lemmas}, since it is somewhat technical.

Our next lemma is similar to Lemma \ref{lem-prob-edge}, but instead of a single set $X_h$ we consider an $L$-tuple of such sets. We would like to show that for each $j$ with probability at least $1-\delta$ at least one of the $L$ corresponding sets $X_h\cup Y_j$ is linearly independent. This is roughly true, but we also have to take into account the possibility of having bad configurations as in Definition \ref{defi-bad}.

\begin{lemma}\label{lem-prob-L-tuple}
Let $H\su \{1,\dots,m\}$ be a subset of size $|H|=L$, and let $j\in \{1,\dots,m\}$. Let us fix any outcomes of $X_h=\{b_{1,h},\dots,b_{n',h}\}$ for all $h\in H$, such that for each $h\in H$ the multi-set $X_h$ is linearly independent. Furthermore, consider any fixed outcomes of the vectors in the sets $Y_1,\dots,Y_{j-1}$, i.e.\ of the vectors $b_{i,1},\dots,b_{i,j-1}$ for all $i=n'+1,\dots,n$. Then, conditional on the outcomes of the vectors in $Y_1,\dots,Y_{j-1}$, subject to the randomness of  $Y_j=\{b_{n'+1,j},\dots,b_{n,j}\}$, with probability at least $1-\delta$ at least one of the following properties is satisfied:
\begin{itemize}
\item[(a)] $X_h\cup Y_j$ is linearly independent for at least one $h\in H$.
\item[(b)] $(H',j)$ is bad for some non-empty subset $H'\su H$.
\end{itemize}
\end{lemma}

We also postpone the proof of Lemma \ref{lem-prob-L-tuple} to Section \ref{sect-lemmas}. Similarly to the deduction of Lemma \ref{lem-high-degrees-left} from Lemma \ref{lem-prob-edge}, Lemma \ref{lem-prob-L-tuple} and the Chernoff bound imply the following.

\begin{lemma}\label{lem-L-tupe-Chernoff}
With probability $1-o(1)$, the following holds: For every subset $H\su \{1,\dots,m\}$ of size $|H|=L$, there are at least $(1-2\delta)m$ different $j\in\{1,\dots,m\}$ which satisfy property (a) or property (b) in Lemma \ref{lem-prob-L-tuple}.
\end{lemma}
\begin{proof}
First, expose the outcomes of $X_1,\dots,X_m$ (i.e.\ the outcomes of the bases $B_1,\dots,B_{n'}$). By Lemma \ref{lem-no-isolated-vert}, with probability $1-o(1)$, each of the multi-sets $X_1,\dots,X_m$ is linearly independent. So let us condition on any such outcome for $X_1,\dots,X_m$.

Now it suffices to prove that (subject to the randomness of $Y_1,\dots,Y_m$) for every subset $H\su \{1,\dots,m\}$ of size $|H|=L$, the probability that there are fewer than $(1-2\delta)m$ different $j\in\{1,\dots,m\}$ satisfying property (a) or (b) in Lemma \ref{lem-prob-L-tuple} is at most $o(1/n^L)$. Indeed, then by the union bound the probability that for some subset $H\su \{1,\dots,m\}$ of size $|H|=L$ there are fewer than $(1-2\delta)m$ such $j$ is at most $m^L\cdot o(1/n^L)=o(1)$.

So let us now consider some $H\su \{1,\dots,m\}$ of size $|H|=L$. By Lemma \ref{lem-prob-L-tuple}, for each $j\in\{1,\dots,m\}$ with probability at least $1-\delta$ the outcome of $Y_j$ is such that (a) or (b) is satisfied, even when conditioning on any outcomes of $Y_1,\dots,Y_{j-1}$. This means that the number of different $j\in\{1,\dots,m\}$ satisfying (a) or (b) is a random variable which stochastically dominates a binomial random variable $Z\sim \operatorname{Bin}(m,1-\delta)$. Thus, the probability that there are fewer than $(1-2\delta)m$ different $j$ satisfying (a) or (b) is at most
\[\Pr[Z<(1-2\delta)m]\leq \exp\left(-2\cdot \frac{((1-\delta)m-(1-2\delta) m)^2}{m}\right)=e^{-2\delta^2 m} \leq e^{-\delta^2\cdot (n-1)}=o(1/n^L),\]
as desired. Here, we again used the Chernoff bound (see e.g.\ \cite[Theorem A.1.4]{alon-spencer}) and $m\in \{\lfloor n/2\rfloor,\lceil n/2\rceil\}$.
\end{proof}

Now, using Lemmas \ref{lem-no-bad-set} and \ref{lem-L-tupe-Chernoff}, we can show that with high probability, for any $L$ distinct vertices on the left side of $G$, the union of their neighborhoods is large.

\begin{lemma}\label{lem-L-tuple-nbhd}
With probability $1-o(1)$, the graph $G$ satisfies the following condition: For any $L$ distinct vertices on the left, the union of their neighborhoods has size at least $(1-2\delta)m$.
\end{lemma}

\begin{proof}
By Lemma \ref{lem-no-bad-set}, with probability $1-o(1)$ none of the pairs $(H',j)$ for any subset $H'\su\{1,\dots,m\}$ of size $|H'|\leq L$ and any $j\in \{1,\dots,m\}$ are bad. Furthermore, with probability $1-o(1)$  the statement in Lemma \ref{lem-L-tupe-Chernoff} holds. But if there are no bad pairs $(H',j)$, then there cannot be any subset $H\su\{1,\dots,m\}$ of size $|H|= L$ and any $j\in \{1,\dots,m\}$ satisfying property (b) in Lemma \ref{lem-prob-L-tuple}. Thus, with probability $1-o(1)$, for every subset $H\su \{1,\dots,m\}$ of size $|H|=L$ there are at least $(1-2\delta)m$ different $j\in\{1,\dots,m\}$ satisfying property (a) in Lemma \ref{lem-prob-L-tuple}. But this precisely means that with probability $1-o(1)$ for every set $H$ of $L$ distinct vertices on the left of $G$ there are at least $(1-2\delta)m$ different vertices $j$ on the right that are in the union of the neighborhoods of the vertices in $H$ on the left.
\end{proof}

Combining Lemmas \ref{lem-high-degrees-left}, \ref{lem-high-degrees-right} and \ref{lem-L-tuple-nbhd}, it is now not hard to show that $G$ satisfies the condition in Halls' marriage theorem with high probability.

\begin{lemma}\label{lem-Hall-condition}
With probability $1-o(1)$, the graph $G$ satisfies the following  condition: For any set $H$ of vertices on the left, the union of their neighborhoods has size at least $|H|$.
\end{lemma}

\begin{proof}
Since we are proving an asymptotic statement, we may assume that $n$ (and hence $m$) is sufficiently large with respect to the fixed value of $c$, and in particular that $2\delta m\geq L$ (recall that $\delta$ and $L$ only depend on $c$).

With probability $1-o(1)$, the graph $G$ satisfies the properties in Lemmas \ref{lem-high-degrees-left}, \ref{lem-high-degrees-right} and \ref{lem-L-tuple-nbhd}. We will show that whenever these properties are satisfied, the graph also satisfies the condition in Lemma \ref{lem-Hall-condition}.

So consider some subset $H$ of the $m$ vertices on the left of $G$. We need to show that the union of the neighborhoods of these vertices has size at least $|H|$, assuming that $G$ satisfies the properties in Lemmas \ref{lem-high-degrees-left}, \ref{lem-high-degrees-right} and \ref{lem-L-tuple-nbhd}.

If $|H|=0$, this is trivially true. If $1\leq |H|\leq 2\delta m$, then the union of the neighborhoods of the vertices in $H$ has size at least $2\delta m\geq |H|$, since by the property in Lemma \ref{lem-high-degrees-left} each of these neighborhoods already has size at least $2\delta m$ individually. If $2\delta m< |H|\leq (1-2\delta) m$, then in particular $|H|> 2\delta m\geq L$ and so by the property in Lemma \ref{lem-L-tuple-nbhd} the union of the neighborhoods of the vertices in $H$ has size at least $(1-2\delta) m\geq |H|$. Finally, if $(1-2\delta) m< |H|\leq m$, then we claim that all $m\geq |H|$ vertices on the right of $G$ must be in the union of the neighborhoods of the vertices in $H$. Indeed, by the property in Lemma \ref{lem-high-degrees-left} each vertex on the right of $G$ has degree at least $2\delta m$ and so it must be adjacent to at least one vertex in $H$ since $|H|>(1-2\delta) m$.
\end{proof}

Whenever the bipartite graph $G$ satisfies the condition in Lemma \ref{lem-Hall-condition}, then by Hall's marriage theorem it has a perfect matching. Thus, by Lemma \ref{lem-Hall-condition}, $G$ has a perfect matching with probability $1-o(1)$. This proves Proposition \ref{propo-1}.

\section{Proof of Lemmas \ref{lem-no-bad-set}, \ref{lem-prob-L-tuple} and \ref{lem-auxiliary}}
\label{sect-lemmas}

In this section, we prove the lemmas whose proofs we previously postponed. First, we prove Lemma \ref{lem-no-bad-set}

\begin{proof}[Proof of Lemma \ref{lem-no-bad-set}]
We need to prove that with probability $1-o(1)$ none of the pairs $(H,Y_j^{(k)})$ for $H\su \{1,\dots,m\}$ of size $1\leq |H|\leq L$, and $j\in \{1,\dots,m\}$ and $k\in \{1,\dots,K\}$ is bad. Let us say that a bad pair $(H,Y_j^{(k)})$ is a minimal, if there is no bad pair $(H',Y_j^{(k)})$ with $\varnothing \subsetneq H'  \subsetneq H$. Note that whenever $(H,Y_j^{(k)})$ is a bad pair, there is some non-empty subset $H'\su H$ such that $(H',Y_j^{(k)})$ is a minimal bad pair.

Thus, it suffices to prove that with probability $1-o(1)$ none of the pairs $(H,Y_j^{(k)})$ for $H\su \{1,\dots,m\}$ of size $1\leq |H|\leq L$, and $j\in \{1,\dots,m\}$ and $k\in \{1,\dots,K\}$ is a minimal bad pair. There are at most $m^L\cdot m\cdot K\leq K\cdot n^{L+1}=O(n^{L+1})$ choices for $H$, $j$ and $k$. Hence it suffices to prove that for any given choice of $H\su \{1,\dots,m\}$ (of size $1\leq |H|\leq L$) and $j\in \{1,\dots,m\}$ and $k\in \{1,\dots,K\}$, the pair $(H,Y_j^{(k)})$ is a minimal bad pair with probability at most $o(1/n^{L+1})$. 

So consider a subset $H\su \{1,\dots,m\}$ of size $1\leq |H|\leq L$, and $j\in \{1,\dots,m\}$ and $k\in \{1,\dots,K\}$. Let $h\in H$ be the largest element of $H$. First, consider the case that $|H|=1$, i.e.\ $H=\{h\}$. We claim that in this case the pair $(H,Y_j^{(k)})$  is never bad (and so in particular never a minimal bad pair). Indeed, we always have
\[\dim\left(\spn(X_h\cup Y_j^{(k)})\right)\leq n-k=n-k\cdot |\{h\}|,\]
since $X_h\cup Y_j^{(k)}$ only consists of $n'+(n-n'-k)=n-k$ vectors. Hence the pair $(H,Y_j^{(k)})$  is never bad if $|H|=1$. So let us from now on assume that $|H|\geq 2$.

Let us expose $Y_j^{(k)}$, as well as $X_1,\dots,X_{h-1}$,. We will show that conditioned on any outcomes for $Y_j^{(k)}$ and $X_1,\dots,X_{h-1}$, subject to the randomness of $X_h$, the pair $(H,Y_j^{(k)})$ is a minimal bad pair with probability at most $o(1/n^{L+1})$. Note that now we already exposed $X_{h'}$ for all $h'\in H\sm\{h\}$ (and note that $H\sm\{h\}\neq \varnothing$ by our assumption that $|H|\geq 2$). If $(H\sm\{h\},Y_j^{(k)})$ is a bad pair, then $(H,Y_j^{(k)})$ cannot be a minimal bad pair. So let us assume that the $(H\sm\{h\},Y_j^{(k)})$ is not bad. Then we have
\[\dim\left(\bigcap_{h'\in H\sm\{h\}}\spn(X_{h'}\cup Y_j^{(k)})\right)\leq n-k\cdot (|H|-1)\]
For simplicity of notation, let $W$ be the space given by the intersection on the left-hand side. Then $\dim W\leq n-k\cdot (|H|-1)$. It suffices to show that the probability of having
\begin{equation}\label{eq-proof-no-bad-pairs}
\dim(W\cap \spn(X_{h}\cup Y_j^{(k)}))>n-k\cdot |H|
\end{equation}
is at most $o(1/n^{L+1})$. Indeed, this would mean that the probability of $(H,Y_j^{(k)})$ being a bad pair, and hence in particular the probability of $(H,Y_j^{(k)})$ being a minimal bad pair, is at most $o(1/n^{L+1})$.

Note that in the case of $\dim W\leq n-k\cdot |H|$, we always have $\dim(W\cap \spn(X_{h}\cup Y_j^{(k)}))\leq \dim W\leq n-k\cdot |H|$, so  (\ref{eq-proof-no-bad-pairs}) never holds. Hence we may from now on assume that
\[n-k\cdot |H|<\dim W\leq n-k\cdot (|H|-1).\]

Note that whenever (\ref{eq-proof-no-bad-pairs}) holds, we must have
\begin{align*}
\dim(W+\spn(X_{h}\cup Y_j^{(k)}))&=\dim(W)+\dim(\spn(X_{h}\cup Y_j^{(k)}))-\dim(W\cap \spn(X_{h}\cup Y_j^{(k)}))\\
&<(n-k\cdot (|H|-1))+(n-k)-(n-k\cdot |H|)=n
\end{align*}
and therefore
\[\dim(W+\spn(X_{h}))=\dim(W+\spn(Y_j^{(k)})+\spn(X_{h}))=\dim(W+\spn(X_{h}\cup Y_j^{(k)}))<n.\]
So it suffices to show that the probability of having $\dim(W+\spn(X_{h}))<n$ (conditioned on the outcomes of $Y_j^{(k)}$ and $X_1,\dots,X_{h-1}$ and subject to the randomness of $X_h$) is at most $o(1/n^{L+1})$.
 
Recall that $W$ is determined by the outcomes of $Y_j^{(k)}$ and $X_1,\dots,X_{h-1}$ that we are conditioning on. Let us now expose $X_h=\{b_{1,h},\dots,b_{n',h}\}$ one vector at a time in this order. At every step $i=1,\dots,n'$, consider the dimension $\dim(W+\spn(b_{1,h},\dots,b_{i,h}))$. By our assumption that $\dim W>n-k\cdot |H|$, this dimension can increase for at most $k\cdot |H|\leq K\cdot L$ steps. Let $I\su \{1,\dots,n'\}$ be the set of indices where the dimension at step $i$ is larger than the dimension at step $i-1$, i.e.\ where $\dim(W+\spn(b_{1,h},\dots,b_{i,h}))>\dim(W+\spn(b_{1,h},\dots,b_{i-1,h}))$. Then $|I|\leq K\cdot L$ and hence  there are at most $(n')^{KL}\leq n^{KL}$ possibilities for the set $I$. If $\dim(W+\spn(X_{h}))<n$, then for every $i\in \{1,\dots,n'\}\sm I$ we must have
\[\dim(W+\spn(b_{1,h},\dots,b_{i,h}))=\dim(W+\spn(b_{1,h},\dots,b_{i-1,h}))\leq \dim(W+\spn(X_{h}))<n.\]
Hence, if $\dim(W+\spn(X_{h}))<n$, then for every $i\in \{1,\dots,n'\}\sm I$ the random vector $b_{i,h}$ must be in the subspace $\dim(W+\spn(b_{1,h},\dots,b_{i-1,h}))$ (which only depends on the previously exposed vectors) and this subspace must have dimension at most $n-1$. For each $i\in \{1,\dots,n'\}\sm I$, conditioned on the previously exposed vectors, the probability that this happens is at most $\alpha_1$ by Lemma \ref{lemma-new-vector-in-subspace} and (\ref{eq-definition-alpha}). Hence the probability to have this property for all $i\in \{1,\dots,n'\}\sm I$ is at most
\[\alpha_1^{n'-|I|}\leq \alpha_1^{\lfloor n/2\rfloor-KL}=o(1/n^{KL+L+1})\]
(here, we used that $0<\alpha_1<1$ only depends on our fixed value of $c$). Since there are at most $n^{KL}$ possibilities for $I$, we can conclude that the total probability of having $\dim(W+\spn(X_{h}))<n$ (conditioned on the outcomes of $Y_j^{(k)}$ and $X_1,\dots,X_{h-1}$) is at most $n^{KL}\cdot o(1/n^{KL+L+1})=o(1/n^{L+1})$, as desired.
\end{proof}

Next, we prove Lemma \ref{lem-prob-L-tuple}.

\begin{proof}[Proof of Lemma \ref{lem-prob-L-tuple}.]
Recall that we fixed outcomes $X_h$ for all $h\in H$ such that for each $h\in H$ the multi-set $X_h$ is linearly independent, and that we are conditioning on fixed outcomes for $Y_1,\dots,Y_{j-1}$. Let us expose the vectors of $Y_j=\{b_{n'+1,j},\dots,b_{n,j}\}$ one vector at a time (these vectors are independent on each other, but for each vector $b_{i,j}$ for $i\in \{n'+1,\dots,n\}$ the probability distribution depends on the fixed outcomes of $b_{i,1},\dots,b_{i,j-1}$ that we are conditioning on). We need to prove that the probability that neither (a) nor (b) holds is at most $\delta$.

If (a) does not hold, then for each $h\in H$, the multi-set $X_h\cup Y_j=X_h\cup \{b_{n'+1,j},\dots,b_{n,j}\}$ must be linearly dependent. So, since each $X_h$ is linearly independent, for each $h\in H$ there must be an index $i\in \{n'+1,\dots,n\}$ such that $b_{i,j}\in \spn(X_h\cup \{b_{n'+1,j},\dots,b_{i-1,j}\})$. 

First, consider the possibility that for some $h\in H$ we have $b_{i,j}\in \spn(X_h\cup \{b_{n'+1,j},\dots,b_{i-1,j}\})$ for an index $i\in \{n'+1,\dots,n-K\}$. For each $h\in H$ and $i\in \{n'+1,\dots,n-K\}$, after having exposed $b_{n'+1,j},\dots,b_{i-1,j}$, the subspace $\spn(X_h\cup \{b_{n'+1,j},\dots,b_{i-1,j}\})$ has dimension at most $i-1$ (since it is spanned by $i-1$ vectors). Hence by Lemma \ref{lemma-new-vector-in-subspace} and (\ref{eq-definition-alpha}), the probability of having $b_{i,j}\in \spn(X_h\cup \{b_{n'+1,j},\dots,b_{i-1,j}\})$ is at most $\alpha_{n-i+1}$. Thus, by a union bound, the overall probability that $b_{i,j}\in \spn(X_h\cup \{b_{n'+1,j},\dots,b_{i-1,j}\})$ for some $h\in H$ and some $i\in \{n'+1,\dots,n-K\}$ is at most
\[|H|\cdot (\alpha_{n-n'}+\alpha_{n-n'+1}+\dots+\alpha_{K+1})< L\cdot (\alpha_K+\alpha_{K+1}+\alpha_{K+2}+\dots)\leq L\cdot\frac{\delta}{2L}=\frac{\delta}{2}\]
by (\ref{eq-summing-up-alpha}).

It remains to bound the probability that for each $h\in H$ there is some $i\in \{n-K+1,\dots,n\}$ such that $b_{i,j}\in \spn(X_h\cup \{b_{n'+1,j},\dots,b_{i-1,j}\})$. More precisely, it suffices to show that the following event $\mathcal{E}$ happens with probability at most $\delta/2$: (b) is not satisfied and for each $h\in H$ there is some $i\in \{n-K+1,\dots,n\}$ such that $b_{i,j}\in \spn(X_h\cup \{b_{n'+1,j},\dots,b_{i-1,j}\})$.

We will show that this event $\mathcal{E}$ happens with probability at most $\delta/2$, even when conditioning on any outcomes of $b_{n'+1,j},\dots,b_{n-K,j}$. So let us expose $b_{n'+1,j},\dots,b_{n-K,j}$ and condition on their outcomes. Now, the only remaining random vectors are $b_{n-K+1,j},\dots,b_{n,j}$ (which we will expose one vector at a time). For every $\ell=1,\dots,K$ and $h\in H$, let $\mathcal{E}_h^{(\ell)}$ denote the event that $b_{n-K+\ell,j}\in \spn(X_h\cup \{b_{n'+1,j},\dots,b_{n-K+\ell-1,j}\})$ and for every non-empty subset $H'\su H$ the pair $(H',Y_j^{(K+1-\ell)})$ is not bad. Now, whenever the event $\mathcal{E}$ happens, for each $h\in H$ there must be some $\ell\in \{1,\dots,K\}$ such that the event $\mathcal{E}_h^{(\ell)}$ happens (indeed, for each $h\in H$ we can take $\ell=i-(n-K)$ for $i$ as in the event $\mathcal{E}$, noting that the pair $(H',Y_j^{(K+1-\ell)})$ is not bad for any non-empty subset $H'\su H$ since (b) is not satisfied).

Thus, it suffices to show that
\begin{equation}\label{eq-to-show-with-aux-lemma}
\Pr\left[\text{for each }h\in H\text{ there is some }\ell\in \{1,\dots,K\}\text{ such that }\mathcal{E}_h^{(\ell)}\text{ holds}\right]\leq \delta/2,
\end{equation}
where the probability is subject only to the randomness $b_{n-K+1,j},\dots,b_{n,j}$ and conditional on the outcomes of all previously exposed vectors. Note that each event $\mathcal{E}_h^{(\ell)}$ only depends on $b_{n-K+1,j},\dots,b_{n-K+\ell,j}$.

Let us apply Lemma \ref{lem-auxiliary} to the random variables $Z_1,\dots,Z_K$ given by $Z_\ell=b_{n-K+\ell,j}$, the events $\mathcal{E}_h^{(\ell)}$ for $h\in H$ and $\ell\in \{1,\dots,K\}$ defined above, and to $\beta_1,\dots,\beta_K$ given by $\beta_\ell=\alpha_{K+1-\ell}$ for $\ell=1,\dots,K$. In order to check the condition in Lemma \ref{lem-auxiliary}, consider any subset $H'\su H$, any $\ell=1,\dots,K$ and any outcomes of $Z_1=b_{n-K+1,j},\dots,Z_{\ell-1}=b_{n-K+\ell-1,j}$. We need to check that
\begin{equation}\label{eq-condition-aux-lemma-check}
\Pr\left[\mathcal{E}_h^{(\ell)}\textnormal{ holds for all }h\in H'\mid b_{n-K+1,j},\dots,b_{n-K+\ell-1,j}\right]\leq (\alpha_{K+1-\ell})^{|H'|}.
\end{equation}
Indeed, first note that this inequality is trivially true if $H'=\varnothing$ (since then the left-hand side is $1$). So let us assume that $H'\su H$ is a non-empty subset. Exposing $b_{n-K+1,j},\dots,b_{n-K+\ell-1,j}$ (and recalling that we are also conditioning on outcomes of the previously exposed vectors) determines  the space
\[\bigcap_{h\in H'}\spn(X_h\cup Y_j^{(K+1-\ell)})=\bigcap_{h\in H'}\spn(X_h\cup \{b_{n'+1,j},\dots,b_{n-K+\ell-1,j}\}).\]
If this space has dimension strictly larger than $n-(K+1-\ell)\cdot |H'|$, then $(H',Y_j^{(K+1-\ell)})$ is bad (by Definition \ref{defi-bad}) and so the event $\mathcal{E}_h^{(\ell)}$ cannot happen (by definition of the event). Hence (\ref{eq-condition-aux-lemma-check}) is also trivially true in this case. So let us assume that the space $\bigcap_{h\in H'}\spn(X_h\cup \{b_{n'+1,j},\dots,b_{n-K+\ell-1,j}\})$ has dimension at most $n-(K+1-\ell)\cdot |H'|$. In order for all of the events $\mathcal{E}_h^{(\ell)}$ for $h\in H'$ to hold, the vector $b_{n-K+\ell,j}$ must be contained in the intersection of $\spn(X_h\cup \{b_{n'+1,j},\dots,b_{n-K+\ell-1,j}\})$ for all $h\in H'$ (again by definition of the events $\mathcal{E}_h^{(\ell)}$). Since this intersection is a space of dimension at most $n-(K+1-\ell)\cdot |H'|$, by Lemma \ref{lemma-new-vector-in-subspace} and (\ref{eq-definition-alpha}) the probability that $b_{n-K+\ell,j}$ is in this intersection is at most
\[\alpha_{(K+1-\ell)\cdot |H'|}\leq (\alpha_{K+1-\ell})^{|H'|}\]
(here, we used Lemma \ref{lem-alpha-powers}). Thus, we have checked the condition (\ref{eq-condition-aux-lemma-check}) in all cases.

So we can indeed apply Lemma \ref{lem-auxiliary}, and conclude that
\begin{align*}
\Pr\left[\text{for each }h\in H\text{ there is some }\ell\in \{1,\dots,K\}\text{ such that }\mathcal{E}_h^{(\ell)}\text{ holds}\right]&\leq \left(1-(1-\beta_1)\dotsm (1-\beta_K)\right)^{|H|}\\
&=\left(1-(1-\alpha_K)\dotsm (1-\alpha_1)\right)^{L}\\
&\leq \left(1-(1-\alpha_1)\cdot (1-\alpha_2)\dotsm \right)^{L}\\
&= (1-3\delta)^{L}\leq \delta/2,
\end{align*}
where in the second-last step we used the definition of $\delta$ in (\ref{eq-defi-delta}) and the last step is by (\ref{eq-def-L}). This proves (\ref{eq-to-show-with-aux-lemma}), as desired.
\end{proof}

Finally, we prove Lemma \ref{lem-auxiliary}.

\begin{proof}[Proof of Lemma \ref{lem-auxiliary}]
We prove the lemma by induction on $K$. For $K=1$, we have 
\begin{multline*}
\Pr\left[\textnormal{for each }h\in H\textnormal{ there is }\ell\in \{1\}\textnormal{ such that }\mathcal{E}_h^{(\ell)}\textnormal{ holds}\right]\\
=\Pr\left[\mathcal{E}_h^{(1)}\textnormal{ holds for all }h\in H\right]\leq (\beta_1)^{|H|}= \left(1-(1-\beta_1)\right)^{|H|},
\end{multline*}
as desired (using the assumption in the lemma statement).

So let us now assume that $K\geq 2$ and that the statement in the lemma holds for $K-1$. Let $T\su H$ be the (random) set
\[T=\{h\in H\mid \textnormal{ there is some }\ell\in \{1,\dots,K-1\}\textnormal{ such that }\mathcal{E}_h^{(\ell)}\textnormal{ holds}\}.\]
Note that $T$ depends only on the outcomes of $Z_1,\dots, Z_{K-1}$ (recall that each of the events $\mathcal{E}_h^{(\ell)}$ for $h\in H$ and $\ell\in \{1,\dots,K-1\}$ depends only on $Z_1,\dots, Z_{K-1}$). Furthermore, note that for every subset $H'\su H$, the induction hypothesis for $K-1$ applied to $H'$ and the events $\mathcal{E}_h^{(\ell)}$ for $h\in H'$ and $\ell=1,\dots,K-1$ gives
\begin{multline}\label{eq-proof-aux-ind-hyp}
\Pr[H'\su T]= \Pr\left[\textnormal{for each }h\in H'\textnormal{ there is }\ell\in \{1,\dots,K-1\}\textnormal{ such that }\mathcal{E}_h^{(\ell)}\textnormal{ holds}\right]\\
\leq \left(1-(1-\beta_1)\dotsm (1-\beta_{K-1})\right)^{|H'|}.
\end{multline}
We have
\begin{multline}\label{eq-lange-gleichung}
\Pr\left[\textnormal{for each }h\in H\textnormal{ there is }\ell\in \{1,\dots,K\}\textnormal{ such that }\mathcal{E}_h^{(\ell)}\textnormal{ holds}\right]\\
=\sum_{J\su H} \Pr[T=J]\cdot \Pr\left[\mathcal{E}_h^{(K)}\text{ holds for all }h\in H\sm J\,\Big\vert\, T=J\right].
\end{multline}
Note that for any $J\su H$ having $T=J$ depends only on the outcomes of $Z_1,\dots, Z_{K-1}$. For any fixed outcomes of $Z_1,\dots, Z_{K-1}$ (and so in particular for any outcomes such that $T=J$ holds), we know by the assumption in the lemma that
\[\Pr\left[\mathcal{E}_h^{(K)}\textnormal{ holds for all }h\in H\sm J\,\Big\vert \,Z_1,\dots,Z_{K-1}\right]\leq (\beta_K)^{|H\sm J|}.\]
Hence the same inequality also holds when condition on $T=J$ instead of conditioning on specific outcomes of $Z_1,\dots, Z_{K-1}$. So for every subset $J\su H$, we can conclude
\begin{align*}
\Pr\left[\mathcal{E}_h^{(K)}\textnormal{ holds for all }h\in H\sm J\,\Big\vert \,T=J\right]&\leq (\beta_K)^{|H\sm J|}\\
&=(\beta_K)^{|H\sm J|}\cdot (\beta_K+(1-\beta_K))^{|J|}\\
&=(\beta_K)^{|H\sm J|}\sum_{H'\su J} (\beta_K)^{|J\sm H'|}\cdot (1-\beta_K)^{|H'|}\\
&=\sum_{H'\su J}(\beta_K)^{|H\sm H'|}\cdot (1-\beta_K)^{|H'|}.
\end{align*}
Combining this with (\ref{eq-lange-gleichung}), we obtain
\begin{align*}
&\Pr\left[\textnormal{for each }h\in H\textnormal{ there is }\ell\in \{1,\dots,K\}\textnormal{ such that }\mathcal{E}_h^{(\ell)}\textnormal{ holds}\right]\\
&\quad\quad\leq \sum_{J\su H} \sum_{H'\su J} \Pr[T=J]\cdot (\beta_K)^{|H\sm H'|} (1-\beta_K)^{|H'|}\\
&\quad\quad= \sum_{H'\su H} \left(\sum_{H'\su J\su H} \Pr[T=J]\right)\cdot (\beta_K)^{|H\sm H'|} (1-\beta_K)^{|H'|}\\
&\quad\quad= \sum_{H'\su H} \Pr[H'\su T]\cdot (\beta_K)^{|H\sm H'|} (1-\beta_K)^{|H'|}\\
&\quad\quad\leq \sum_{H'\su H} \left(1-(1-\beta_1)\dotsm (1-\beta_{K-1})\right)^{|H'|}\cdot (\beta_K)^{|H\sm H'|} (1-\beta_K)^{|H'|}\\
&\quad\quad=  \Big(\left(1-(1-\beta_1)\dotsm (1-\beta_{K-1})\right)\cdot (1-\beta_K)+\beta_K\Big)^{|H|}\\
&\quad\quad=  \left(1-(1-\beta_1)\dotsm (1-\beta_K)\right)^{|H|},
\end{align*}
where for the second inequality sign we used (\ref{eq-proof-aux-ind-hyp}).
\end{proof}

\section{Concluding remarks}
\label{sect-concluding-remarks}

The most obvious open question is, of course, to prove Rota's basis conjecture (either for vector spaces as stated in Conjecture \ref{conjecture-rota}, or more generally for matroids). But there are also still many interesting questions concerning the conjecture for random bases in various families of matroids.

Graphic matroids are an important class of matroids, and Rota's basis conjecture is still wide open in this case. In the spirit of the results in this paper, one may ask whether Rota's basis conjecture holds (with high probability) for random maximal acyclic sets in a given graph $G$ (more precisely, assuming that $G$ has $n$ vertices and $m$ components, then every maximal acyclic set $F\su E(G)$ has size $n-m$, and we can consider independent random sets $F_1,\dots,F_{n-m}\su E(G)$ each chosen uniformly at random from the collection of all maximal acyclic set $F\su E(G)$). A particularly natural case of this problem is the case where $G$ is the complete graph on $n$ vertices. This leads to the following question.

\begin{problem}\label{ref-problem-complete-graph}
Let $K_n$ be the complete graph on $n$ vertices, and let $F_1,\dots,F_{n-1}\su E(G)$ be independent uniformly random spanning trees of $K_n$. Prove that with probability $1-o(1)$, Rota's basis conjecture is satisfied for $F_1,\dots,F_{n-1}\su E(G)$, i.e.\ that with probability $1-o(1)$ the multi-set $F_1\cup\dots\cup F_{n-1}$ can be partitioned into $n-1$ spanning trees of $K_n$ which are transversal with respect to the original spanning trees $F_1,\dots,F_{n-1}$.
\end{problem}

Here, a spanning tree $F\su F_1\cup \dots\cup F_{n-1}$ is called transversal with respect to $F_1,\dots,F_{n-1}$, if it contains exactly one edge from each of $F_1,\dots,F_{n-1}$ (where $F_1,\dots,F_{n-1}$ are interpreted as subsets of the multi-set $F_1\cup \dots \cup F_{n-1}$ in the natural way, and if an edge $e$ appears in two or more of the spanning trees $F_1,\dots,F_{n-1}$, then the multiple copies of $e$ in the multi-set $F_1\cup \dots \cup F_{n-1}$ are distinguished by which set $F_i$ they came from).

It is well-known that graphic matroids are representable (i.e.\ they can be represented by linear independence relations in vector spaces). Indeed, given a graph with vertex set $\{1,\dots,n\}$, for each edge $ij$ with $1\leq i<j\leq n$, one can consider the vector $e_i-e_j\in \mathbb{R}^n$ (where $e_1,\dots,e_n$ denotes the standard basis of $\mathbb{R}^n$). Then a set of edges is acyclic if and only if the corresponding set of vectors in $\mathbb{R}^n$ is linearly independent.

Via this connection, it is easy to see that Problem \ref{ref-problem-complete-graph} is actually equivalent to proving Rota's basis conjecture for random bases in the subset
\[T=\{e_i-e_j\mid 1\leq i<j\leq n\}\su \mathbb{R}^n.\]
More precisely, Problem \ref{ref-problem-complete-graph} is equivalent to proving that for independent random bases $B_1,\dots,B_n\su T$ of $\mathbb{R}^n$ (where each $B_i$ is chosen uniformly at random from the collection of bases of $\mathbb{R}^n$ that are subsets of $T$), Rota's basis conjecture holds with probability $1-o(1)$. In other words, Problem \ref{ref-problem-complete-graph} is equivalent to proving a version of Theorem \ref{thm-3} for the set $T$ defined above.

Note that this set $T$ is not $c$-dispersed for any fixed $0<c<1$ (for large $n$), hence Theorem \ref{thm-3} does not apply to this set $T$. It would be interesting to generalize Theorem  \ref{thm-3} to a wider class of sets $T$, and in particular to resolve Problem \ref{ref-problem-complete-graph}.

\end{document}